\documentclass[12pt]{amsart}
\usepackage[a4paper, bottom=30mm, left=30mm, right=30mm, top=30mm]{geometry}
\usepackage{graphicx}
\usepackage{amsmath}
\usepackage{amsxtra}
\usepackage{amscd}
\usepackage{amsthm}
\usepackage{amsfonts}
\usepackage{amssymb}
\usepackage{eucal}
\usepackage{color}
\usepackage[numeric]{amsrefs}
\usepackage[all]{xy}
\usepackage[colorlinks=true, pdfstartview=FitV, linkcolor=blue, citecolor=blue, urlcolor=blue]{hyperref}
\input colordvi

\newcommand{\nc}{\newcommand}
\newcommand{\rc}{\renewcommand}

\newcommand{\mf}[1]{{\mathfrak{#1}}}
\newcommand{\mc}[1]{{\mathcal{#1}}}
\newcommand{\mb}[1]{{\mathbb{#1}}}
\newcommand{\mrm}[1]{{\mathrm{#1}}}

\newcommand{\bslash}{
  \mathchoice
    {\reflectbox{$\displaystyle{/}$}}
    {\reflectbox{$\textstyle{/}$}}
    {\reflectbox{$\scriptstyle{/}$}}
    {\reflectbox{$\scriptscriptstyle{/}$}}}

\nc{\nl}{\newline}

\makeatletter
\newcommand*\bigcdot@[2]{\mathbin{\vcenter{\hbox{\scalebox{#2}{$\m@th#1\bullet$}}}}}
\newcommand*\bigcdot{\mathpalette\bigcdot@{.5}}
\makeatother

\theoremstyle{plain}
\newtheorem{thm}{Theorem}[section]
\newtheorem*{thm*}{Theorem} 
\newtheorem{prop}[thm]{Proposition}
\newtheorem{lem}[thm]{Lemma}

\numberwithin{equation}{section}

\DeclareMathOperator{\spec}{Spec}

\DeclareMathOperator{\Sym}{Sym}

\nc{\on}{\operatorname}

\nc{\al}{{\alpha }}
\nc{\be}{{\beta }}
\nc{\ga}{{\gamma }}
\nc{\de}{{\delta }}
\nc{\del}{{\partial }}
\nc{\ep}{{\varepsilon }}
\nc{\vap}{{\epsilon }}
\nc{\rh}{{\rho}}

\nc{\ze}{{\zeta }}
\nc{\et}{{\eta }}
\rc{\th}{{\theta }}

\nc{\vth}{{\vartheta }}

\nc{\io}{{\iota }}
\nc{\ka}{{\kappa }}
\nc{\la}{{\lambda }}
\nc{\vrho}{{\varrho}}
\nc{\si}{{\sigma }}
\nc{\ups}{{\upsilon }}
\nc{\vphi}{{\varphi }}
\nc{\om}{{\omega }}

\nc{\Ga}{{\Gamma }}
\nc{\De}{{\Delta }}
\nc{\nab}{{\nabla}}
\nc{\Th}{{\Theta }}
\nc{\La}{{\Lambda }}
\nc{\Si}{{\Sigma }}
\nc{\Ups}{{\Upsilon }}
\nc{\Om}{{\Omega }}

\renewcommand\si{\sigma}
\renewcommand\l{\lambda}

\renewcommand\th{\theta}

\nc{\fA}{{\mathfrak A}}
\nc{\fB}{{\mathfrak B}}
\nc{\fC}{{\mathfrak C}}
\nc{\fD}{{\mathfrak D}}
\nc{\fE}{{\mathfrak E}}
\nc{\fF}{{\mathfrak F}}
\nc{\fG}{{\mathfrak G}}
\nc{\fH}{{\mathfrak H}}
\nc{\fI}{{\mathfrak I}}
\nc{\fJ}{{\mathfrak J}}
\nc{\fK}{{\mathfrak K}}
\nc{\fL}{{\mathfrak L}}
\nc{\fM}{{\mathfrak M}}
\nc{\fN}{{\mathfrak N}}
\nc{\fO}{{\mathfrak O}}
\nc{\fP}{{\mathfrak P}}
\nc{\fQ}{{\mathfrak Q}}
\nc{\fR}{{\mathfrak R}}
\nc{\fS}{{\mathfrak S}}
\nc{\fT}{{\mathfrak T}}
\nc{\fU}{{\mathfrak U}}
\nc{\fV}{{\mathfrak V}}
\nc{\fW}{{\mathfrak W}}
\nc{\fZ}{{\mathfrak Z}}
\nc{\fX}{{\mathfrak X}}
\nc{\fY}{{\mathfrak Y}}
\nc{\fa}{{\mathfrak a}}
\nc{\fb}{{\mathfrak b}}
\nc{\fc}{{\mathfrak c}}
\nc{\fd}{{\mathfrak d}}
\nc{\fe}{{\mathfrak e}}
\nc{\ff}{{\mathfrak f}}
\nc{\fg}{{\mathfrak g}}
\nc{\fh}{{\mathfrak h}}
\nc{\fiI}{{\mathfrak i}}  
\nc{\ffi}{{\mathfrak i}}  
\nc{\fj}{{\mathfrak j}}
\nc{\fk}{{\mathfrak k}}
\nc{\fl}{{\mathfrak{l}}}
\nc{\fm}{{\mathfrak m}}
\nc{\fn}{{\mathfrak n}}
\nc{\fo}{{\mathfrak o}}
\nc{\fp}{{\mathfrak p}}
\nc{\fq}{{\mathfrak q}}
\nc{\fr}{{\mathfrak r}}
\nc{\fs}{{\mathfrak s}}
\nc{\ft}{{\mathfrak t}}
\nc{\fu}{{\mathfrak u}}
\nc{\fv}{{\mathfrak v}}
\nc{\fw}{{\mathfrak w}}
\nc{\fz}{{\mathfrak z}}
\nc{\fx}{{\mathfrak x}}
\nc{\fy}{{\mathfrak y}}

\nc{\bA}{{\mathbb A}}
\nc{\bB}{{\mathbb B}}
\nc{\bC}{{\mathbb C}}
\nc{\bD}{{\mathbb D}}
\nc{\bE}{{\mathbb E}}
\nc{\bF}{{\mathbb F}}
\nc{\bG}{{\mathbb G}}
\nc{\bH}{{\mathbb H}}
\nc{\bI}{{\mathbb I}}
\nc{\bJ}{{\mathbb J}}
\nc{\bK}{{\mathbb K}}
\nc{\bL}{{\mathbb L}}
\nc{\bM}{{\mathbb M}}
\nc{\bN}{{\mathbb N}}
\nc{\bO}{{\mathbb O}}
\nc{\bP}{{\mathbb P}}
\nc{\bQ}{{\mathbb Q}}
\nc{\bR}{{\mathbb R}}
\nc{\bS}{{\mathbb S}}
\nc{\bT}{{\mathbb T}}
\nc{\bU}{{\mathbb U}}
\nc{\bV}{{\mathbb V}}
\nc{\bW}{{\mathbb W}}
\nc{\bZ}{{\mathbb Z}}
\nc{\bX}{{\mathbb X}}
\nc{\bY}{{\mathbb Y}}

\nc{\cA}{{\mathcal A}}
\nc{\cB}{{\mathcal B}}
\nc{\cC}{{\mathcal C}}
\nc{\cD}{{\mathcal D}}
\nc{\cE}{{\mathcal E}}
\nc{\cF}{{\mathcal F}}
\nc{\cH}{{\mathcal H}}
\nc{\cI}{{\mathcal I}}
\nc{\cJ}{{\mathcal J}}
\nc{\cK}{{\mathcal K}}
\nc{\cL}{{\mathcal L}}
\nc{\cM}{{\mathcal M}}
\nc{\cN}{{\mathcal N}}
\nc{\cO}{{\mathcal O}}
\nc{\cP}{{\mathcal P}}
\nc{\cQ}{{\mathcal Q}}
\nc{\cR}{{\mathcal R}}
\nc{\cS}{{\mathcal S}}
\nc{\cT}{{\mathcal T}}
\nc{\cU}{{\mathcal U}}
\nc{\cV}{{\mathcal V}}
\nc{\cW}{{\mathcal W}}
\nc{\cZ}{{\mathcal Z}}
\nc{\cX}{{\mathcal X}}
\nc{\cY}{{\mathcal Y}}

\newcommand\codim{\operatorname{codim}}

\newcommand{\oh}{\operatorname{H}}

\newcommand{\Hom}{\operatorname{Hom}}

\newcommand{\gr}{\operatorname{gr}}

\newcommand{\HC}{\operatorname{HC}}
\newcommand{\HCH}{\operatorname{MHM}}

\renewcommand{\codim}{\operatorname{codim}}

\nc{\crhom}{{\operatorname{R}\cH\mathit o\mathit m}}

\newcommand{\gobble}[1]{}
  \newcommand{\rangeref}[2]{%
    \ref{#1}--\afterassignment\gobble\fam 0\ref{#2}%
  }

\begin{document}

\title{Hodge filtrations on tempered Hodge modules}

\date{March 2023}

\author{Dougal Davis}\address[DD]{School of Mathematics and Statistics, University of Melbourne, VIC 3010, Australia}
\email{dougal.davis1@unimelb.edu.au}
\thanks{DD was supported by the EPSRC programme grant EP/R034826/1 and the ARC grant FL200100141}

\author{Kari Vilonen}\address[KV]{School of Mathematics and Statistics, University of Melbourne, VIC 3010, Australia, also Department of Mathematics and Statistics, University of Helsinki, Helsinki, Finland}
\email{kari.vilonen@unimelb.edu.au, kari.vilonen@helsinki.fi}
\thanks{KV was supported in part by the ARC grants DP180101445,  FL200100141 and the Academy of Finland}


\begin{abstract}
We show that the Hodge filtration of a tempered Hodge module is generated by the lowest piece of its Hodge filtration. As a consequence, we prove the main conjecture of~\cite{schmid-vilonen} in the special case of tempered representations of real reductive Lie groups.
\end{abstract}

\maketitle

\section{Introduction}

This paper forms part of an ongoing series that aims to study unitary representations of real reductive groups using tools from the theory of mixed Hodge modules in algebraic geometry, implementing a program initiated by Schmid and the second author \cite{schmid-vilonen}.

Let $G_\mb{R}$ be a real reductive group, i.e., the Lie group of real points of a reductive algebraic group defined over $\mb{R}$. A classical open problem in representation theory is to determine the set of unitary irreducible representations of $G_\mb{R}$, often called the unitary dual. This problem has turned out to be remarkably difficult. While there are solutions in certain special cases, such as for complex classical groups \cite{barbasch} using ad hoc methods and for real exceptional groups \cite{atlas, ALTV} using computer calculations, a completely satisfactory description of the unitary dual for general $G_\mb{R}$ is yet to be found.

The program proposed in \cite{schmid-vilonen} aims to place this problem in a more conceptual context. The starting point is the localization theory of Beilinson and Bernstein \cite{BB1}, which relates representations of $G_\mb{R}$ to certain geometric gadgets, called twisted $\mc{D}$-modules, defined over the flag variety of the complexification $G$ of $G_\mb{R}$. This is now a standard tool in representation theory: for example, it was used to great effect by Beilinson and Bernstein in their proof of the Jantzen conjectures \cite{BB}, and by Lusztig and Vogan to give character formulas for irreducible representations \cite{LV} of real groups. The main idea of the present program is that natural enhancements of the geometric side, provided by the theory of mixed Hodge modules, should yield information about unitarity.

The precise connection between Hodge modules and unitarity is as follows. First, Beilinson-Bernstein localization implies that any irreducible representation of $G_\mb{R}$, in its realization as a Harish-Chandra module, arises as the global sections of an equivariant twisted $\mc{D}$-module on the flag variety of $G$. We will refer to this twisted $\mc{D}$-module as the corresponding Harish-Chandra sheaf. As long as the infinitesimal character is real (a minor restriction), the general theory of mixed Hodge modules endows the irreducible Harish-Chandra sheaf with a canonical pure Hodge module structure. This in turn endows the Harish-Chandra module with a filtration (the Hodge filtration) and a Hermitian form (the polarization), both of geometric origin. The main conjecture of \cite{schmid-vilonen} is that these together form a polarized Hodge structure on the Harish-Chandra module\footnote{In fact, this was conjectured to hold for all pure Hodge modules, not just in the Harish-Chandra case.}; in particular, this would imply that the signature of the polarization, and hence unitarity, can be read off from the Hodge filtration.

Since the Hodge filtrations on mixed Hodge modules satisfy many excellent properties, it is expected that this approach to unitarity would be much more tractable than dealing with Hermitian forms directly. In particular, it is hoped that the good functoriality of the Hodge filtration will reveal new structural features of the unitary dual not accessible by other methods.

In a previous paper \cite{DV}, we made the first steps towards implementing this program, confirming that results about unitarity can indeed be obtained from mixed Hodge modules by deriving a key theorem in \cite{ALTV} as a consequence of a stronger result.

In this paper, we study the detailed structure of the Hodge filtration in the special case of tempered Harish-Chandra sheaves. A tempered Harish-Chandra sheaf is one corresponding to a tempered representation. Tempered representations, which are defined by integrability properties of their matrix coefficients, are known to be unitary on general grounds. They play a fundamental role in the representation theory of real groups, especially in the unitarity algorithm of \cite{ALTV}. While the Hodge filtration on a general Harish-Chandra sheaf can be complicated, we prove the following result (Theorem \ref{main} in the main body of the paper) in this special case.

\begin{thm*}
The Hodge filtration of a tempered Hodge module is generated by the lowest piece of its Hodge filtration as a filtered $\cD$-module. 
\end{thm*}
In other words, the Hodge filtration on a tempered Hodge module is, in some sense, as simple as possible.

As is explained in section~\ref{thc}, a general tempered Hodge module is a mixture of the Hodge module of a tempered spherical principal series representation of a split group, and a Hodge module associated with a closed orbit in a partial flag variety. In the latter case, the proof follows easily from the definition of push-forward for filtered $\cD$-modules. In the former case a crucial ingredient is the fact that the minimal $K$-types lie in the lowest piece of the Hodge filtration, one of the main results in~\cite{DV}.

As a consequence of the theorem above we obtain
\begin{thm*}
The main conjecture in~\cite{schmid-vilonen} holds in the tempered case.
\end{thm*}
This result is stated as Theorem~\ref{signature} and we refer the reader there for a more precise statement. In a sequel to this paper, we intend to use it as the base case of an inductive proof of the conjecture for general Harish-Chandra modules.

The authors thank Wilfried Schmid for his contributions to this paper.

\section{Tempered Harish-Chandra sheaves}
\label{thc}

It is convenient for us to follow the basic set up of~\cite{DV}, which we recall here briefly. 
We will work in the context of Harish-Chandra modules. Let us fix a complex reductive group $G$ and an involution $\theta$ of $G$. We write $K=G^\theta$ for the fixed point group. We will always use lower case Gothic letters to denote the corresponding Lie algebras. On the level of Lie algebras we have the Cartan decomposition $\fg = \fk \oplus \fp$ into eigenspaces of $\theta$. If $B$ is a Borel and $N$ is its unipotent radical then $H=B/N$, the universal Cartan, is independent of the choice of $B$ and comes equipped with a canonical root system. In what follows we will always consider the roots in $B$ to be negative. 

We write $\HC(\fg,K)$ for the category of Harish-Chandra modules of the pair $(\fg,K)$, and $\HC(\fg,K)_\l$ for the
full subcategory of Harish-Chandra modules with infinitesimal character $\chi_\l$ associated with $\l\in\fh^*$ under the Harish-Chandra homomorphism. For convenience, we deviate from the notation in~\cite{DV} and use  
Harish-Chandra's convention for the Harish-Chandra homomorphism. In particular, $\HC(\fg,K)_\rho$ contains the trivial representation, and $\l=0$ corresponds to the most singular infinitesimal character and is the center of the action of the Weyl group $W$. 

Let us write $\cB=G/B$ for the flag manifold of $G$. Associated to $\l$ we have the sheaf of twisted differential operators $\cD_\l$ on $\cB$. We write $\HC(\cD_\lambda,K)$ for the category of Harish-Chandra sheaves, i.e., for the category of $K$-equivariant $\cD_\l$-modules on $\cB$. If the parameter $\lambda$ is dominant then, according to Beilinson-Bernstein, each irreducible Harish-Chandra module $M$ is obtained as global sections of a unique irreducible Harish-Chandra sheaf $\cM$. We call an irreducible Harish-Chandra sheaf \emph{tempered} if the associated representation is. The tempered Harish-Chandra modules were first classified in~\cite{KZ1982}. A geometric classification, which we recall below, is given in~\cite{HMSWII}.

An irreducible Harish-Chandra sheaf $\cM$ is an intermediate extension of a rank one equivariant $\lambda$-twisted local system $\gamma$ on a $K$-orbit $Q$, i.e., $\cM=j_{!*}\gamma$ for $j:Q \to \cB$ the inclusion. We call $\cM=j_{!*}\gamma$  \emph{clean} if it coincides with $j_!\gamma$ and hence with $j_*\gamma$. Associated to the orbit $Q$, we may choose, uniquely up to $K$-conjugacy, a $\theta$-stable Cartan $T$ with a fixed point in $Q$. We decompose $\ft$ under $\theta$ into its eigenspaces as $\ft=\fa\oplus\fc$ with $\fa$ being the $(-1)$-eigenspace and $\fc$ the $(+1)$-eigenspace. Identifying $T$ with $H$ via the chosen fixed point, $\gamma$ is specified by a one dimensional Harish-Chandra module for the pair $(\mf{h}, T^\theta)$, i.e., by a pair $\lambda\in\fh^*$, $\Lambda:T^\theta\to \bC^*$ such that $d\Lambda + \rho= \lambda|_\fc$.

We now impose the condition that the infinitesimal character $\l\in\fh^*_\bR$ is real; here $\fh^*_\bR =\bR\otimes_\bZ \mb{X}^*(H)$ for $\mb{X}^*(H)=\Hom(H,\bC^*)$ the character lattice. All other tempered representations are obtained by moving the parameter $\nu = \lambda|_{\mf{a}^*}$ in the imaginary direction. 
Under this condition, the classification states that $\cM$ is tempered if and only if
\begin{equation*}
\la|_\fa = 0 \ \ \  \text{and} \ \ \ \text{$\cM=j_{!*}\gamma$ is clean}\,.
\end{equation*}
Thus, in particular, $d\Lambda = \lambda - \rho$ is integral, and hence cleanness of $\mc{M}$ imposes a further condition on the orbit $Q$: for any complex positive root $\alpha$, the root $\theta\alpha$ must also be positive.

Let us conclude by remarking that orbits $Q$ with the above property have the following simple geometric structure. Let $\operatorname{Z}_\fg(\fc)=\fl\oplus \fc$ and let us write
\begin{equation}
\fv\ = \ \bigoplus_{\substack{\alpha \in \Phi^-(\fg,\ft) \\ \alpha|_\fc\neq 0}}\fg_\alpha\,.
\end{equation}
This gives us a parabolic
\begin{equation*}
\fq_L\ = \ \fl\oplus \fc\oplus \fv\,.
\end{equation*}
Let us first consider the projection
\begin{equation*}
p:\cB \to \cB_{Q_L}\,,
\end{equation*}
where $\cB_{Q_L}$ denotes the generalized flag manifold of parabolics of type $\fq_L$.
Then $\fq_L$ is $\theta$-stable and hence the image of $Q$ in $\cB_{Q_L}$ is closed. The fiber of $p$ is the flag manifold $\cB_L$ of $L$. Furthermore, $\fa$ is the Lie algebra of a maximal torus in $L$ and hence $(L,\theta)$ is split and the orbit $Q\cap \cB_L$ is the open orbit in $\cB_L$.

\section{Tempered Hodge modules}

As  in~\cites{DV,schmid-vilonen}, we work in this section in the context of twisted mixed Hodge modules. In the general set-up in~\cite{DV} we work in the context of complex Hodge theory as in~\cite{SS}, but in the tempered case we treat here working with the mixed Hodge modules of Saito~\cites{S1,S2} would also suffice.

The category $\HC(\cD_\l,K )$ has a mixed Hodge module version $\HCH(\cD_\l,K )$ (denoted by $\mrm{MHM}_{\lambda - \rho}(K \bslash \mc{B})$ in \cite{DV}) if the parameter $\l\in\fh^*_\bR$, which we assume from now on.

An irreducible Harish-Chandra sheaf $\cM = j_{!*}\gamma$ supported on an orbit $Q$ has a unique lift to $\HCH(\cD_\l, K)$ of weight $\dim Q$ with Hodge filtration given by
\begin{equation*}
\begin{aligned}
F_p \gamma \ &= \ \begin{cases}\ 0\  \ &p<0\ ,
\\
\ \gamma \ \ \ &p\geq 0\ .\end{cases}
\end{aligned}
\end{equation*}

Assume that the infinitesimal character $\la$ is dominant and real. We can now state our main result:
\begin{thm}
\label{main}
Let us write $c=\codim Q$. Then the Hodge filtration of an irreducible tempered Harish-Chandra sheaf $j_{!*}\gamma$ supported on $Q$ is generated by $F_cj_{!*}\gamma$. 
\end{thm}

We remark that Theorem \ref{main} is a local statement about the generation of the Hodge filtration of a tempered Harish-Chandra sheaf as a $\mc{D}_\lambda$-module. In the case of spherical principal series for split groups only, we also prove the global statement that the Hodge filtration of the corresponding tempered Harish-Chandra module is generated by its lowest piece (Theorem \ref{thm:split}), but we do not know this in general.

We also prove the following result, which is the main conjecture of \cite{schmid-vilonen} in the special case of tempered Harish-Chandra sheaves.

\begin{thm}
\label{signature}
In the context of Theorem \ref{main}, let $S$ be the polarization on $j_{!*}\gamma$ and $\Gamma(S)$ the induced Hermitian form on the irreducible tempered $(\fg, K)$-module $V := \Gamma(j_{!*}\gamma)$ (see e.g., \cite[\S 4.3]{DV}). Then the form
\[ \Gamma(S)|_{F_p V \cap (F_{p-1} V)^\perp} \]
is $(-1)^{p - c}$-definite for all $p$.
\end{thm}

\section{Tempered Hodge modules for split groups}
\label{split}

In this section we consider the special case of a Hodge module $\cM$ associated with a tempered spherical principal series representation of a split group. Such an $\cM$ is an intermediate extension of a local system on the open orbit $Q$. By the considerations of section~\ref{thc}, we know that the infinitesimal character is zero (i.e., $\cM$ is a $\cD_0$-module) and that the extension is clean. Thus, $\cM$ is self dual as a Hodge module and hence as a filtered $\cD_0$-module. We also recall that $\cD_0^{op} = \cD_0$, so the filtered dual is given simply by the filtered derived hom
\[ \mb{D}(\mc{M}, F) = \operatorname{RHom}_{(\mc{D}_0, F)}((\mc{M}, F), (\mc{D}_0, F))[\dim \mc{B}] \]
without the need for a twist by a line bundle.

The corresponding spherical Harish-Chandra module $M=\Gamma(\cB,\cM)$ is isomorphic, as a $(\fg,K)$-module, to $U(\fg)_0 \otimes_{U(\fk)} \bC_0$, where $U(\fg)_0$ is the quotient of $U(\fg)$ through which the center acts by infinitesimal character $\chi_0$ and $\bC_0$ is the trivial representation of $K$. We have a corresponding description of $\cM$ as
\begin{equation*}
\cM \ = \ \cD_0 \otimes_{U(\fk)} \bC_0\,.
\end{equation*}
Using this description, we may view $\cM$ as a filtered module with the filtration induced from $\cD_0$. We denote this filtration by $F'_\bullet \cM$. This filtration is generated by the rank one $\cO_X$-module $F'_0\cM$. Let us identify $\fg$ with $\fg^*$, write $\cN$  for the nilpotent cone in $\fg$ and write $\mu: T^*\mc{B}\to \cN$ for the moment map. We then have 

\begin{thm}
\label{thm:split}
The filtration $F'$ on $\cM$ coincides with the Hodge filtration $F$ and  $\gr^{F}_\bullet\cM = \mu^*(\cO_\fp)$. Furthermore, $\oh^k(\cB,\gr^{F}_\bullet\cM)=0$ for $k>0$ and $\Gamma(\cB,\gr^{F}_\bullet\cM) = \cO_{\cN\cap \fp}$.
\end{thm}

The rest of this section is devoted to the proof of this theorem. We begin with
\begin{lem}
The filtered module $(\cM,F')$ is filtered self dual and $\gr^{F'}_\bullet\cM = \mu^*(\cO_\fp)$ is Cohen-Macaulay.
\end{lem}

It is perhaps helpful to recall that the associated graded of the Hodge filtration of any Hodge module is Cohen-Macaulay. The Cohen-Macaulay condition is necessary to have a good notion of filtered dual. 

\begin{proof}
We begin with some preparatory statements. Observe that we have $\mu^{-1}(\fp)=T^*_K\cB$, the union of conormal bundles of $K$-orbits. Thus, because the group is split, $\mu^{-1}(\fp)$ is of pure codimension
\begin{equation*}
\codim_{T^*\cB} \mu^{-1}(\fp) = \dim K.
\end{equation*}
We conclude that the scheme $\mu^{-1}(\mf{p})$, being the preimage of a codimension $\dim K$ complete intersection under a map of smooth varieties, is itself a complete intersection and its structure sheaf $\mu^*(\cO_\fp)$ has the Koszul resolution
\begin{equation}
\label{koszul}
\dots \to \cO_{T^*\cB}\{i\}  \otimes \wedge^i \fk \to  \cO_{T^*\cB}\{i-1\} \otimes \wedge^{i-1} \fk \to \dots\to \cO_{T^*\cB}\{1\}  \otimes \fk \to \cO_{T^*\cB} \,,
\end{equation}
where $\{\cdot\}$ denotes a grading shift.

Let us now consider the filtered module $(\cM,F')$. It has the following filtered resolution
\begin{equation}
\label{resolution}
\dots \to \cD_0\{i\}  \otimes \wedge^i \fk \to  \cD_0\{i-1\} \otimes \wedge^{i-1} \fk \to \dots\to \cD_0\{1\}  \otimes \fk \to \cD_0 \,. 
\end{equation}
Here $\cD_0\{i\}$ denotes the filtered module $\cD_0$ with filtration shifted so that it begins in degree $i$. To verify that the complex above is a resolution we pass to the associated graded complex. The graded resolution coincides with~\eqref{koszul}. Thus we conclude that~\eqref{resolution} is a filtered resolution of $(\cM,F')$ and, furthermore that  $\gr^{F'}_\bullet = \mu^*(\cO_\fp)$ 

We now form the filtered dual $\bD(\cM,F')$ by making use of the resolution above. We obtain, writing $n=\dim(X)=\dim \fk$:
\begin{equation*}
\dots \to \cD_0\{-i\}  \otimes \wedge^i \fk^* \to  \dots\to \cD_0\{-n+1\} \otimes \wedge^{n-1} \fk^*\to \cD_0\{-n\} \otimes \wedge^n \fk^* \,. 
\end{equation*}
Thus, we conclude that $\bD(\cM,F')=(\cM,F'\{-n\})$, i.e., that $(\cM,F')$ is filtered self dual. 
\end{proof}

\begin{lem}
If $F'_0\cM \subset F_0\cM$ then $F'_i\cM = F_i\cM$ for all $i$.
\end{lem}
\begin{proof}
Our hypothesis implies that the identity map is a morphism of filtered modules $(\cM,F') \to (\cM,F)$, i.e., that $F'_i\cM \subset F_i\cM$ for all $i$. As both $(\cM,F')$ and $(\cM,F)$ are self dual filtered modules we obtain a morphism of filtered modules $(\cM,F) \to (\cM,F')$, i.e., that $F_i\cM \subset F'_i\cM$ for all $i$. Thus $F'_i\cM = F_i\cM$ for all $i$.
\end{proof}

\begin{proof}[Proof of Theorem~\ref{thm:split}]
By the above lemmas, it remains to show that $F'_0\cM \subset F_0\cM$. The module $F'_0\cM $ is generated by the minimal $K$-type $\bC_0$. But by~\cite[Theorem 4.5]{DV} the minimal $K$-type lies in $F_0\cM$, so this concludes the proof of the first part of the theorem. 

We will now give a proof of the second part. 
We will show that the global sections of the Koszul resolution \eqref{koszul} give a resolution for $\mc{O}_{\mc{N} \cap \mf{p}}$. Since $\oh^k(\mc{B}, \mc{O}_{T^*\mc{B}}) = 0$ for $k > 0$, this implies the statement.

As the group is split the Cartan subspace $\fa$ of $\fp$ is also a Cartan for $\fg$ and the little Weyl group of $\fa$ is the full Weyl group of $\fg$. Thus the restriction map
\begin{equation}
\label{invariant}
\bC[\fg]^G \xrightarrow {\ \sim \ }  \bC[\fp]^K 
\end{equation}
is an isomorphism. Consider the diagram
\[
\xymatrix{
\mf{p} \ar[r] \ar[d] & \mf{g} \ar[d] \\
\spec \mb{C}[\mf{p}]^K \ar@{=}[r] & \spec \mb{C}[\mf{g}]^G.
}
\]
By Kostant \cite[Theorems 0.1 and 0.2]{K} and Kostant-Rallis~\cite[Theorems 14 and 15]{KR}, the vertical maps are flat and their scheme-theoretic fibers $\mc{N}$ and $\mc{N} \cap \mf{p}$ are reduced complete intersections. We see therefore that
\begin{equation*}
\cO_{\cN\cap\fp} \ =\  \cO_\cN \otimes_{\cO_\fg}^\bL \cO_\fp
\end{equation*}
has the graded Koszul resolution
\begin{equation}
\label{koszul2}
\dots \to \cO_{\mc{N}}\{i\}  \otimes \wedge^i \fk \to  \cO_{\mc{N}}\{i-1\} \otimes \wedge^{i-1} \fk \to \dots\to \cO_{\mc{N}}\{1\}  \otimes \fk \to \cO_{\mc{N}}. \,
\end{equation}
But \eqref{koszul2} is given by global sections of \eqref{koszul} so we are done.
\end{proof}

\section{Proof of Theorem~\ref{main}}
\label{tempered}

Let us now consider a general irreducible tempered Hodge module $\cM$. By the considerations in section~\ref{thc}, we have $\cM=j_{!*}\gamma$ where $\gamma$ is a clean local system on an orbit $Q$ with the special properties specified in that section. Using the notation in section~\ref{thc} we write $S=p(Q)$ for the closed $K$-orbit in the partial flag variety. Then $\bar Q=p^{-1}(S)$ and we have a $K$-equivariant smooth fibration $\bar p: \bar Q \to S$. We further write $i: \bar Q \to \mc{B}$ for the closed embedding and $\tilde j: Q \to \bar Q$ for the open embedding.

Let us start by considering $\cN= \tilde j_{!*}(\gamma)$. Let us consider the fiber $\cB_L$ of $\bar p$ and the restrictions $\cN|_{\cB_L}$ and $\gamma_L = \gamma|_{Q\cap \mc{B}_L}$. As the restriction is non-characteristic, we have $\cN|_{\cB_L}=\tilde j_{!*}\gamma_L$. So, after we adjust the cohomological shift and the weights, the Hodge module $\cN|_{\cB_L}$ is the tempered spherical principal series sheaf considered in section~\ref{split} for $(\fl,K_L)$, and so we know that its Hodge filtration is generated by $F_0\cN|_{\cB_L}$. Thus the same is true for $\cN$ and we conclude that its Hodge filtration is generated by $F_0\cN$.

Let us write $\cI$ for the ideal sheaf of $\bar Q$ in $\cB$. As $\cM=i_*\cN$ and $i$ is an inclusion of a closed smooth subvariety 
\[
\cM = i_*\cN = (\cD_{\cB, \lambda}/\cD_{\cB, \lambda}\cI) \otimes_{\cO_{\bar Q}} \cN\otimes_{\cO_{\bar Q}}\omega_{\bar Q/\cB} \,.
\]
Let us write $c=\codim{Q}$. Then, by the definition of filtered proper push-forwards, we have
\begin{equation*}
 F_p\cM =F_pi_*\cN = \sum_{r+k\leq p-c}   (F_k\mc{D}_{\mc{B}, \lambda}/ F_k\mc{D}_{\mc{B}, \lambda}\,\cI )\otimes_{\cO_{\bar Q}}F_r\cN\otimes_{\cO_{\bar Q}}\omega_{\bar Q/\cB}.
 \end{equation*}
From this formula and the fact that the Hodge filtration of $\cN$ is generated by $F_0\cN$ we conclude that the Hodge filtration of $\cM$ is generated by $F_c\cM$.
 
\section{Proof of Theorem \ref{signature}}

We will prove Theorem \ref{signature} by appealing to the known unitarity of tempered Harish-Chandra modules. We first recall the relationship between compact and non-compact real forms for $G$ and the corresponding invariant Hermitian forms, cf., \cite[\S 12]{ALTV}.

Recall that we have fixed a complex reductive group $G$ with an involution $\theta$. We choose also a compact real form $U_\mb{R} \subset G$ invariant under $\theta$. Denoting by $\sigma_c : G \to G$ the complex conjugation with respect to $U_\mb{R}$, the involution $\theta$ determines another real form $G_\mb{R} = G^{\theta \sigma_c}$. The corresponding real Lie algebras are $\mf{u}_\mb{R} = \mf{g}^{\sigma_c}$ and $\mf{g}_\mb{R} = \mf{g}^{\theta\sigma_c}$. A key property of tempered Harish-Chandra modules is that they always admit positive definite $\mf{g}_\mb{R}$-invariant Hermitian forms.

The link between the $\mf{g}_\mb{R}$-invariant forms and the polarizations is as follows. First, since the flag variety $\mc{B}$ and the universal Cartan $H$ are canonically associated with $G$, the involution $\theta : G \to G$ induces compatible involutions on both. We will write $\delta : H \to H$ for the induced involution on $H$. We remark that $\delta$ preserves the positive roots in $\mb{X}^*(H)$, and if $Q \subset \mc{B}$ is a closed $K$-orbit and $T$ is a $\theta$-stable Cartan with a fixed point in $Q$, then $\delta$ agrees with $\theta|_T$ under the corresponding identification $T \cong H$. The involution $\theta : \mc{B} \to \mc{B}$ also lifts non-canonically to the $H$-torsor $\tilde{\mc{B}}$, intertwining the action of $\delta$ on $H$; we fix such a lift in what follows.

For any $(\mc{D}_\lambda, K)$-module $\mc{M}$, the pullback $\theta^*\mc{M}$ is a $(\mc{D}_{\delta \lambda}, K)$-module, and we have an isomorphism
\[ \Gamma(\mc{M}) \overset{\sim}\to \Gamma(\theta^*\mc{M}) \]
intertwining $\theta : U(\mf{g}) \to U(\mf{g})$. If $\mc{M} = j_{!*}\gamma$ is an irreducible tempered module and $\lambda \in \mf{h}^*_\mb{R}$ is real, then we have $\delta \lambda = \lambda$ and $\theta^*j_{!*}\gamma \cong j_{!*}\gamma$. We will fix such an isomorphism so that the induced map
\[ \theta : \Gamma(j_{!*}\gamma) \to \Gamma(j_{!*}\gamma) \]
is equal to the identity on the (unique) minimal $K$-type. In particular, the above map is an involution.

Suppose now that $\lambda \in \mf{h}^*_\mb{R}$, and $j_{!*}\gamma$ is an irreducible tempered $(\mc{D}_\lambda, K)$-module, equipped with its polarization $S$ as a pure Hodge module. We therefore have the $\mf{u}_\mb{R}$-invariant form $\Gamma(S)$ on $\Gamma(j_{!*}\gamma)$; it follows immediately that the form
\[ \langle u, v \rangle := \Gamma(S)(u, \theta v) \]
is $\mf{g}_\mb{R}$-invariant. Since the tempered representation $\Gamma(j_{!*}\gamma)$ is unitary as a $\mf{g}_\mb{R}$-module and $\Gamma(S)$ is positive definite on the minimal $K$-type by \cite[Theorem 4.3 and Proposition 4.7]{DV}, we deduce the following.

\begin{prop} \label{prop:theta signature}
For $\epsilon = \pm 1$, the polarization $\Gamma(S)$ is $\epsilon$-definite on the $\epsilon$-eigenspace $\Gamma(j_{!*}\gamma)^{\epsilon\theta}$.
\end{prop}

We now claim the following.

\begin{prop} \label{prop:hodge parity}
The associated graded of the Hodge filtration on $\Gamma(j_{!*}\gamma)^\theta$ (resp., $\Gamma(j_{!*}\gamma)^{-\theta}$) is concentrated in degrees $2k + c$ (resp., $2k + 1 + c$) for $k \in \mb{Z}_{\geq 0}$.
\end{prop}

\begin{proof}[Proof of Theorem \ref{signature}]
Follows immediately from Propositions \ref{prop:theta signature} and \ref{prop:hodge parity}.
\end{proof}

\begin{proof}[Proof of Proposition \ref{prop:hodge parity}]
Consider first the case of a tempered spherical principal series for a split group as in section \ref{split}. By Theorem \ref{thm:split}, we have
\[ \gr^F_\bullet \Gamma(j_{!*}\gamma) = \mc{O}_{\mc{N} \cap \mf{p}} \]
is naturally a graded quotient of $\Sym(\mf{p})$. Since $\theta$ acts on $\mf{p}$ with eigenvalue $-1$ by definition the result in this case follows.

For the general tempered case, consider as in section \ref{tempered} the smooth fibration
\[ \bar{p} \colon \bar{Q} \to S\]
and the Hodge module $\tilde{j}_{!*}\gamma$ on $\bar{Q}$. We have
\begin{equation} \label{eq:hodge parity 1}
\Gamma(\mc{B}, \gr^F_\bullet j_{!*}\gamma) = \Gamma(S, \mc{F}\{c\}),
\end{equation}
where
\[ \mc{F} = \Sym(\mc{N}_{S/\mc{B}_{Q_L}}) \otimes \omega_{S/\mc{B}_{Q_L}} \otimes \bar{p}_{\bigcdot} \gr^F_\bullet \tilde{j}_{!*}\gamma,\]
$\{c\}$ denotes a grading shift, and $\bar{p}_{\bigcdot}$ the sheaf-theoretic pushforward. Since $S$ is closed, it is fixed by $\theta$ pointwise, so $\theta$ acts on the sheaf $\mc{F}$. Since $\bar{p}_{\bigcdot} \tilde{j}_{!*}\gamma$ is fiberwise a tempered spherical principal series for $L$, the $+1$ (resp., $-1$) eigenspace of $\gr^F_\bullet \tilde{j}_{!*}\gamma$ is concentrated in only even (resp., odd) degrees as shown above. Moreover, $\theta$ acts on the normal bundle $\mc{N}_{S/\mc{B}_{Q_L}}$ with eigenvalue $-1$, so the above is also true for $\mc{F}$. The proposition now follows by \eqref{eq:hodge parity 1}.
\end{proof}

\end{document}